\documentclass[11pt, oneside,reqno]{amsart}
\usepackage{amsfonts,amsmath,mathabx,fullpage,amssymb,amsthm,tikz,mathtools,lmodern,paralist,bm}
\usepackage{mathrsfs,hyperref,stmaryrd,enumitem,tikz-cd,shuffle}
\usetikzlibrary{patterns}

\usepackage{color}

\definecolor{orange}{rgb}{0.898, 0.621, 0.0}
\definecolor{skyblue}{rgb}{0.336, 0.703, 0.910}
\definecolor{bluishgreen}{rgb}{0, 0.617, 0.449}
\definecolor{yellow}{rgb}{0.937, 0.890, 0.258}
\definecolor{blue}{rgb}{0, 0.445, 0.695}
\definecolor{red}{rgb}{0.832, 0.367, 0}
\definecolor{purple}{rgb}{0.797, 0.473, 0.652}

\DeclareMathOperator{\conv}{conv}
\DeclareMathOperator{\cone}{cone}

\newcommand{\Rr}{\mathbb{R}}

\newcommand{\sB}{\mathrm{B}}
\newcommand{\sC}{\mathrm{C}}
\newcommand{\sH}{\mathrm{H}}
\newcommand{\sK}{\mathrm{K}}
\newcommand{\sL}{\mathrm{L}}
\newcommand{\sS}{\mathrm{S}}
\newcommand{\sP}{\mathrm{P}}

\newcommand{\sF}{\mathrm{F}}
\newcommand{\sQ}{\mathrm{Q}}
\newcommand{\sM}{\mathrm{M}}

\newcommand{\bA}{\mathbf{A}}

\newcommand{\bC}{\mathbf{C}}
\newcommand{\bK}{\mathbf{K}}
\newcommand{\bP}{\mathbf{P}}

\newcommand{\bS}{\mathbf{S}}

\newcommand{\ba}{\mathbf{a}}
\newcommand{\bb}{\mathbf{b}}
\newcommand{\bc}{\mathbf{c}}

\newcommand{\xx}{\mathbf{x}}
\newcommand{\yy}{\mathbf{y}}
\newcommand{\zz}{\mathbf{z}}

\newcommand{\pp}{\mathbf{p}}
\newcommand{\qq}{\mathbf{q}}
\newcommand{\uu}{\mathbf{u}}
\newcommand{\vv}{\mathbf{v}}

\DeclareMathOperator{\interior}{int}
\DeclareMathOperator{\closure}{cl}
\DeclareMathOperator{\projection}{Pr}
\DeclareMathOperator{\vis}{V}
\DeclareMathOperator{\witness}{Wit}

\newcommand{\0}{\emptyset}

\numberwithin{equation}{section}
\newtheorem{theorem}{Theorem}
\newtheorem{mainthm}{Theorem}
\newtheorem{lemma}[theorem]{Lemma} 

\newtheorem{proposition}[theorem]{Proposition}
\newtheorem{corollary}[theorem]{Corollary}

\theoremstyle{definition}
\numberwithin{theorem}{section}
\newtheorem{definition}[theorem]{Definition}
\newtheorem{example}[theorem]{Example}
\newtheorem{remark}[theorem]{Remark}

\newtheorem{question}[theorem]{Question}

\newcommand{\dllin}{\ar@{-}[dl]}
\newcommand{\drlin}{\ar@{-}[dr]}
\newcommand{\dlin}{\ar@{-}[d]}

\newcommand{\cC}{\mathcal{C}}

\newcommand{\cH}{\mathcal{H}}

\newcommand{\cK}{\mathcal{K}}

\author[F.~Castillo]{Federico Castillo}
\address[F.~Castillo]{Departamento de Matem\'aticas, Pontificia Universidad Cat\'olica de Chile, Santiago, Chile}
\email{federico.castillo@mat.uc.cl}
\author[J.~Doolittle]{Joseph Doolittle}
\address[J.~Doolittle]{Institut f\"ur Geometrie, Technische Universit\"at Graz, Austria}
\email{jdoolittle@tugraz.at}
\author[J.~Samper]{Jos\'e A.\ Samper}
\address[J.~Samper]{Departamento de Matem\'aticas, Pontificia Universidad Cat\'olica de Chile, Santiago, Chile}
\email{jsamper@mat.uc.cl}

\begin{document}
	\begin{abstract}
		It is well-known since the time of the Greeks that two disjoint circles in the plane have four common tangent lines. Cappell et al. proved a generalization of this fact for properly separated strictly convex bodies in higher dimensions.
		We have shown that the same generalization applies for arbitrary convex bodies.
		When the number of convex sets involved is equal to the dimension, we obtain an alternative combinatorial proof of Bisztriczky's theorem on the number of common tangents to $d$ separated convex bodies in $\Rr^d$.
		
	\end{abstract}
	
	\title{ Common tangents to convex bodies}
	
	\maketitle

	\section{Introduction}

	It a known since the times of Euclid and Apollonius that two disjoint circles have four common tangents. In fact they had explicit constructions with straightedge and compass to describe these lines.
	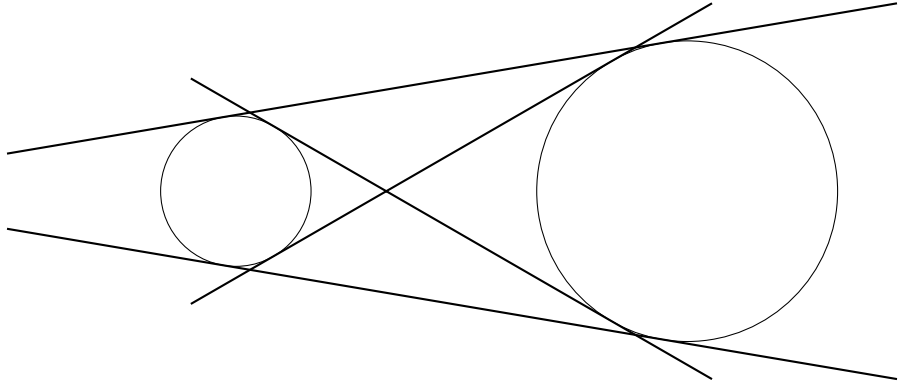
\begin{figure}[ht]
		\begin{tikzpicture}
			\draw (-2,0) circle (1cm);
			\draw (4,0) circle (2cm);
			\draw[thick] (210:3cm)--(30:5cm) (150:3cm)--(330:5cm);
			
			\begin{scope}[xshift=-8cm]
				\draw[thick] (9.594:3cm) -- (9.594:15cm);
				\draw[thick] (-9.594:3cm) -- (-9.594:15cm);
			\end{scope}
			
		\end{tikzpicture}
		\caption{Four tangents to two circles.}
		\label{fig:two}
	\end{figure}
	Note that the circles need to be disjoint for these four tangents to exist; the number of common tangents can be any integer less than four depending on whether the circles are internally/externally tangent, intersecting, or nested.
	This paper is a generalization of this result in the realm of convex geometry; we describe the set of common tangent hyperplanes to suitably separated convex bodies.
	
	
	In order to properly state our main theorem, we first define the separation we use.
	Let $\cK^d$ be the set of convex bodies (compact, convex, with non-empty interior) in $\Rr^d$.
	We say that a family $\bS=\{\sS_1,\dots,\sS_m\}\subset \cK^d$  is \emph{strongly separated} if for every subset $I\subseteq [m]$ there exists an affine hyperplane $\sH$ such that $\bigcup_{i\in I} \sS_i\subseteq \sH^-$ and $\bigcup_{i\notin I} \sS_i\subseteq \sH^+$.
	The set $\mathcal{H}^d$ of all hyperplanes in $\Rr^d$ is parametrized by the real projective space $\mathbb{RP}^{d}$.
	Let $\mathcal{T}(\bS)\subset \mathcal{H}^d$ be the set of hyperplanes that are tangent to a family $\bS$ and contain the entire family on the same side. Our main theorem is the following.
	
	\begin{mainthm}\label{thm:main}
		Let $\bK=\{\sK_1,\dots,\sK_m\}\subset \cK^d$ be a family of strongly separated convex bodies in $\Rr^d$ where $m\leq d$.
		The set $\mathcal{T}(\bK)$ is homeomorphic to the sphere $\mathbb{S}^{d-m}$.
	\end{mainthm}
	
	This is a generalization of a theorem of Cappell, Goodman, Pach, Pollack, Sharir and Wenger \cite[Theorem 2]{cappell1994common}.
	They proved Theorem \ref{thm:main} in the case where every convex body is strictly convex.
	We remark that our proof strategy is different to the original technique in \cite{cappell1994common}.
	There, the strict convexity is used  in an essential way to show that the geometric locus of the tangent hyperplanes form a manifold with boundary as the bodies are appropriately deformed.
	We instead build on intuition from polytope theory and Bruggeser and Mani's approach to shellability \cite[Section 4]{mani} to provide an inductive argument. 	
	We consider the convex hull of all the bodies and interpret the separation conditions as visibility conditions in the polar convex body.
	This idea allows us to formulate a dual problem that implies our theorem and which allows us to reduce dimension properly. 
	
	
	Our theorem applies to arbitrary convex bodies and in particular to polytopes, in this case the set of tangent hyperplanes is a polyhedral complex (Theorem \ref{thm:polytopes}) combinatorially equivalent to the boundary of a polytope.
	We also remark that our results give an alternative proof of Bisztricksky's theorem \cite{bisztriczky1990separated} that there are exactly $2^d$ tangent hyperplanes to $d$ strongly separated convex bodies in $\Rr^d$.
	This generalizes the result mention at the beginning, see Figure \ref{fig:two}.
	The core of this proof, as explained in \cite{bisztriczky1990separated}, is to show that there exist exactly two common tangents to the $d$ strongly separated convex bodies in $\Rr^d$ with every body on the same side of both tangents, so that these hyperplanes sandwich the whole family.
	This is the case $m=d$ of Theorem \ref{thm:main}.
	We remark that \cite{cappell1994common} generalizes their version of Theorem \ref{thm:main} to arbitrary convex bodies when $m=d$ in order to prove Bisztricksky's theorem.

	

	It may be worth to note that there is another topological proof of Bisztricksky's theorem by Lewis, von Hohenbalken, and Klee \cite{lewis1996common} using Kakutani's extension of Brouwer’s fixed point theorem. 
    Our proof is elementary, but the arguments are subtle: in the introduction of \cite{lewis1996common}, the authors mention that Bisztriczky had announced that his original proof of the Theorem was insufficient.
	In any case, there are at least different proofs and also several generalizations of this theorem, see e.g.  \cite{barany2008slicing}, \cite{karasev2009theorems},  \cite{kincses2008topological}, and \cite{klee1997appollonius}.

	Finally, the subject of common tangents to multiple objects has been considered from an algebraic point of view; the circles in the original Greek problem have been replaced by real quadrics or convex semialgebraic sets, and the lines have been replaced by $k$-planes. See for example \cite{borcea2006common}, \cite{kozhasov2020number}, \cite{sottile2002lines}, and \cite{sottile2005real}.
	It would be interesting to know if the algebraic results for general $k$-planes also extend to convex bodies.

	\subsection*{Acknowledgements} 
	We thank Mauricio Bustamante, Florian Frick, Frank Sottile, and Giancarlo Urzúa for helpful discussions.
	This project began in the Max Planck Institute for Mathematics in the Sciences, Leipzig.
	The first author thanks the Institut f\"ur Mathematik,
	Arbeitsgruppe Diskrete Geometrie at FU Berlin for their hospitality.
	The first-named author was partially supported by the FONDECYT Regular grant 1221133.
	The second-named author was supported by the Austrian Science Fund FWF, grant P 33278. The third-named author was partially supported by the FONDECYT Iniciación grant 11221076.
	
	\section{Preliminaries and notation}
	
	An (affine) hyperplane in $\Rr^d$ can be written as $\sH_{\uu,\alpha}:=\{\xx\in\Rr^d~:~\langle \xx,\uu\rangle=\alpha\}$, where $\uu\in \Rr^d$ is a nonzero vector and $\alpha\in \Rr$ is any real scalar.
	Every affine hyperplane defines two open halfspaces (we informally call them \emph{sides}):
	\[
	\sH^+_{\uu,\alpha}:=\{\xx\in\Rr^d~:~\langle \xx,\uu\rangle>\alpha\}, \text{ and}\quad \sH^-_{\uu,\alpha}:=\{\xx\in\Rr^d~:~\langle \xx,\uu\rangle<\alpha\}.
	\]
	The positive and negative parts are exchanged if we replace $\uu$ by $-\uu$ and $\alpha
	$ by $-\alpha$.
	We denote their closures $\sH^\geq_{\uu,\alpha}=\sH^+_{\uu,\alpha}\cup \sH_{\uu,\alpha}$
	and $\sH^\leq_{\uu,\alpha}=\sH^-_{\uu,\alpha}\cup \sH_{\uu,\alpha}$, and we often omit the subscripts $\uu,\alpha$.
	We say a hyperplane $\sH$ is tangent to (or supporting of) a set $S\subset\Rr^d$ if $S\cap \sH$ is nonempty and $S$ is contained in one of the two closed halfspaces defined by $\sH$.
	
	An affine linear subspace $\sL$ of dimension $k$ is called a $k-$flat.
	Notice that there is a unique $d-$flat in $\Rr^d$ which is the whole space.
	For $k<d$ a $k-$flat $\sL$ is tangent at $\sK$ if $\sL$ is contained in a tangent hyperplane to $\sK$ and $\sL\cap\sK$ is nonempty.
	By convention, the ambient space $\Rr^d$, the unique $d-$flat, is tangent to every convex body. 
	
	We denote the topological interior of a set $S$ by $\interior(S)$ and its closure by $\closure(S)$.
	A convex body $\sK\subseteq \Rr^d$ is a compact convex set with $\interior(\sK)\neq\0$.
	The set of all convex bodies in $\Rr^d$ is denoted $\cK^d$.
	Two convex bodies $\sK_1$ and $\sK_2$ are separated if there exists a hyperplane $\sH$ such that $\sK_1\subseteq\sH^{\geq}$ and $\sK_2\subseteq\sH^{\leq}$.
	If $\geq,\leq$ can be replaced by $+,-$, then the separation is called strict.
	A convex body $\sK$ is said to be strictly convex if its intersection with each of its tangent hyperplanes is a singleton.
	Unit closed balls are strictly convex, polytopes are not.
	
	In \cite{bisztriczky1990separated} and \cite{cappell1994common}; they used a different definition of \emph{separated}, they call a family of subsets $\bS=\{\sS_1,\dots,\sS_m\}\subseteq\cK^d$ separated if for every $n$-dimensional affine subspace, with $1\leq n\leq d-2$, intersects at most $n+1$ members of $\bS$. 
	This notion of separation is equivalent to what we define as \emph{strong separation} (see \cite[Lemma 1]{bisztriczky1990separated} for one direction).

	Sometimes it is more useful to use cones instead of convex bodies
	A family of cones $\bC=\{\sC_1,\dots,\sC_k\}$ in $\Rr^d$ is said to be \emph{acyclic} if there is no linear dependence on $\cone(\bC)$ with all coefficients positive.
	See \cite[Chapter 6.2]{ziegler} for more information.
	
	Any family of convex bodies $\bK=\{\sK_1,\dots,\sK_k\}$ in $\Rr^d$ can be turned into an acyclic family of cones in $\Rr^{d+1}$ by defining $\sC_i:=\cone(\sK_i,1)$ over each convex body.
	We call this construction the linearization of $\bK$.
	Conversely, for any acyclic family of cones $\bC=\{\sC_1,\dots,\sC_k\}$ in $\Rr^{d+1}$, there exists some hyperplane $\sH$ that intersects the interior of each $\sC_i$.
	The family defined by $\sK_i:=\sC_i\cap\sH$ consists of convex bodies in $\sH\cong\Rr^d$.
	
	\subsection{Polarity}
	
	We always assume that the origin is in the interior of the convex body under consideration.
	We define the \emph{polar} of the body $\sK$ to be
	\[
	\sK^\circ := \{\yy\in\Rr^d~:~\langle\yy,\xx\rangle\leq 1 \, \text{for every } \xx\in \sK\}.
	\]
	Intuitively, the polar is the set of valid inequalities.
	Polarity allow us to exchange boundary points and supporting hyperplanes.
	\begin{itemize}
		\item We have $\yy\in\partial\sK^\circ$ if and only if there exists an $\xx\in\sK$ with $\langle\yy,\xx\rangle=1$.
		\item The hyperplane $\{\yy\in\Rr^d~:~\langle\yy,\xx\rangle\leq 1\}$ is supporting for $\sK^\circ$ if and only if $\xx\in\partial\sK$. 
	\end{itemize}

	\subsection{Visibility}
	Let $\sH\subset\Rr^d$ be a hyperplane and $\ba \notin \sH$.
	We define the projection away from $\ba$ to $\sH$ as the function
	\begin{equation}
		\projection_{\ba,\sH}(\cdot):\Rr^d\to \sH\cong\Rr^{d-1},\quad \projection_{\ba,\sH}(\xx)=\textrm{aff. span}(\ba,\xx)\cap\sH.
	\end{equation}
	The range of this function is equal to $\sH$ which can be identified with $\Rr^{d-1}$.
	Its domain is $\Rr^{d}\backslash \sH_\ba$ where $\sH_\ba$ is the hyperplane parallel to $\sH$ passing through $\ba$.

	\begin{definition}
		Let $\sK\subseteq\cK^d$ and $\ba\notin\sK$.
		We say that $\qq\in\sK$ is \emph{visible from} $\ba$ if $\sK\cap[\qq,\ba]=\{\qq\}$, otherwise we say $\qq$ is \emph{covisible from} $\ba$.
		The set of visible points in $\sK$ from $\ba$ is denoted $\vis_\ba(\sK)$.
		We call a subset $S$ of $\sK$ visible if there exists $\ba$ such that $\vis_\ba(\sK)=S$.
		The closure of the complement of a visible set is called a covisible set. 
	\end{definition}
	
	\begin{remark}
		It is often convenient to not concern ourselves with the difference between visibility and covisibility and treat them on equal footing.
		Any visible (resp. covisible) set can be transformed into a covisible (resp. visible) set by a projective transformation.
	\end{remark}

	Most of the time we can check visibility by the existence of certain tangent hyperplanes.
	
	\begin{lemma}\label{lem:visibility}
		Let $\sK\subseteq\cK^d$ and $\ba\notin\sK$.
		We have that if $\qq\in\interior(\vis_\ba(\sK))$ then there exists a hyperplane $\sH$ tangent at $\qq$ strictly separating $\interior(\sK)$ and $\ba$.
	\end{lemma}
	
	\begin{proof}
		Assume $\qq$ is visible from $\ba$ then the convex half-open segment $[\ba,\qq)$  is disjoint from $\sK$.
		By the Separation Theorem \cite[Theorem 1.3.7]{schneider} there exists a hyperplane $\sH$ separating $\sK$ and $[\ba,\qq)$  which contains $\qq$ by construction.
		Note that $\ba\notin\sH$ since otherwise $\qq\in\partial\vis_\ba(\sK)$.
		So $\sH$ is the desired hyperplane.
	\end{proof}
	\begin{remark}
		The converse of Lemma \ref{lem:visibility} is not true. If $\sK$ is a triangle in the plane and $\ba$ is a point that sees a single edge $[\qq_1,\qq_2]$ then both $\qq_1,\qq_2$ are not in the interior of the visibility region but the span of the segment is a hyperplane satisfying the conditions of the Lemma.
	\end{remark}
	
	We define the witness set for $\sK\in\cK^d$ and $S\subset\partial \sK$:
	\begin{equation}\label{eq:witness}
		\witness_\sK(S):=\{\ba\in\Rr^d~:~\vis_\ba(\sK)=S\}.
	\end{equation}
	Lemma \ref{lem:visibility} allow us to describe witness sets.
	\begin{lemma}\label{lem:halfspaces}
	Let $\sK\in\cK^d$ and $S\subset\partial \sK$.
	We define a collection of halfspaces $\cH$ as follows:
	\begin{enumerate}
		\item It contains $\sH^>$ for every hyperplane $\sH$ tangent to $\sK$ at $\qq\in\interior(S)$ and such that $\sK\subset\sH^\leq$.
		\item It contains $\sH^\leq$ for every hyperplane $\sH$ tangent to $\sK$ at $\qq\in\partial\sK\backslash\interior(S)$ and such that $\sK\subset\sH^\leq$.
	\end{enumerate}
	We have that $\bigcap\cH=\witness_{\sK}(S)$.
	In particular, witness sets are convex.
	\end{lemma}

	\begin{proof}
		Let $\ba\in\bigcap\cH$ then the first condition one ensures that $\interior(S)\subseteq\vis_\ba(\sK)$.
		Lemma \ref{lem:visibility} implies that $\partial\sK\backslash\interior(S)$ is disjoint from $\interior(\vis_\ba(\sK)))$.
		Since $\vis_\ba(\sK)$ is always closed these conditions imply that it is equal to $S$, so that $\bigcap\cH\subseteq\witness_{\sK}(S)$.
		
		Conversely for any $\ba\in\witness_{sK}(S)$ we have by Lemma \ref{lem:visibility} that $\ba$ is in all halfspaces of $\cH$, and so $\witness_{\sK}(S)\subseteq \bigcap\cH$.
	\end{proof}
	
	We shall use the following result due to Ewald, Larman, Rogers \cite{ewald1970directions} (see also \cite[Section 2.3]{schneider}) that was later generalized by Zalgaller \cite{zalgaller1972k}.
	
	\begin{theorem}[Ewald, Larman, Rogers ]
		\label{thm:key}
		Let $\sK\in\cK^d$. The set of vectors $\uu\in\mathbb{S}^{d-1}$ such that there is a $1-$flat parallel to $\uu$ tangent to $\sK$ at more than one point has measure zero in the unit sphere $\mathbb{S}^{d-1}$.
		As a consequence, a generic orthogonal projection $\pi$ into a hyperplane $\sH$ induces an homeomorphism between $\partial\pi(\sK)$ and its preimage.
	\end{theorem}
	
	The following Lemma is the key tool for our proof of Theorem \ref{thm:main}.

	\begin{proposition}\label{prop:projection}
		Let $\ba\notin\sK$ be a generic point and $\sH$ hyperplane strictly separating $\ba$ from $\sK$, then $\projection_{\ba,\sH}(\cdot)$ maps $\vis_\ba(\sK)$ into a convex set $\sM\subset\sH$ and  induces an homeomorphism between $\partial \vis_\ba(\sK)$ and $\partial \sM$.
		
		Let $\bb\in\Rr^d$ be another point such that $\vis_\ba(\sK)\subset\vis_\bb(\sK)$ and $\bb'=\projection_{\ba,\sH}(\bb)\notin\sM$.
		Let $\sH_\ba$ be the hyperplane parallel to $\sH$ passing through $\ba$.
		If $\bb$ is in the same side as $\sK$ with respect to $\sH_\ba$, then
		\begin{equation}\label{eq:condition}
			\qq\in\partial\bigg(\vis_\bb(\sK)\backslash\vis_\ba(\sK)\bigg)\cap\partial\vis_\ba(\sK)\Longrightarrow
			\qq'\in \vis_{\bb'}(\sM),
		\end{equation}
		
		where $\qq'=\projection_{\ba,\sH}(\qq)$.
		On the other hand, if $\bb$ is in the opposite side as $\sK$ with respect to $\sH_\ba$, then
		\begin{equation}\label{eq:condition2}
			\qq\in\partial\bigg(\vis_\bb(\sK)\backslash\vis_\ba(\sK)\bigg)\cap\partial\vis_\ba(\sK)\Longrightarrow
			\qq'\in \sM\backslash\vis_{\bb'}(\sM).
		\end{equation}
	\end{proposition}

	\begin{proof}
		For the first part we apply a projective transformation sending $\ba$ to a point at infinity.
		In this case $\projection_{\ba,\sH}$ is an orthogonal projection to a generic $\sH$.
		Theorem \ref{thm:key} implies the first part of the statement.
		In particular we can assume that for every $\qq\in\partial \vis_\ba(\sK)$ the line $\sL$ spanned by $\ba$ and $\qq$ is tangent at $\sK$ only at $\qq$.
		Now we prove the second part.
		
		We reduce to a two dimensional case.
		If $\bb\in\sL$, consider any $2-$flat $\sF$ containing $\sL$ intersecting the interior of $\sK$
		If $\bb\notin\sL$, consider the $2-$flat $\sF$ spanned by $\sL$ and $\bb$.
		We restrict to $\sF\cong \Rr^2$ which we identify with the xy-plane. 
		We obtain a convex body $\sK'=\sK\cap\sF$ with a point $\qq$ in the boundary and the line $\sL$ that we can assume to be the y-axis.
		If a point $\pp\in\sK\cap\sL$ is visible from $\bb$ in $\Rr^d$, then it is also visible when restricted to $\sF$.
		The condition of Equation \eqref{eq:condition} implies that for every open ball $\sB$ centered in $\qq$ the sets $\sB\cap\vis_\ba(\sK)$ and $\sB\cap(\vis_\bb(\sK)\backslash\vis_\ba(\sK))$ are both nonempty and are contained in $\sB\cap\vis_\bb(\sK)$.
		When restricting to $\Rr^2$ these conditions imply that $\bb$ sees points arbitrarily close to $\qq$ from both directions (up and down).
		It follows that that $\bb$ must be strictly to the right side of $\sL$, see Figure \ref{fig:2d} (in particular, it follows that $\bb$ cannot be in $\sL$).
		
		Note that supporting hyperplanes at $\qq'$ in the projection are in bijection with supporting hyperplanes in $\Rr^d$ containing $\sL$.
		Since $\vis_\bb(\sK)$ is closed it is enough to assume that $\qq\in\interior(\vis_\bb(\sK))$ on the left hand side of Equations \eqref{eq:condition}--\eqref{eq:condition2}.
		This implies that any hyperplane $\sH'$ containing $\sL$ will not contain $\bb$, and there are two cases:
		\begin{enumerate}
			\item If $\bb$ is and $\sK$ are on the same side relative to $\sH_\ba$, then in the projection $\qq'$ is visible from $\bb'$.
			\item If $\bb$ is and $\sK$ are on opposite sides relative to $\sH_\ba$, then $\qq'$ is covisible from $\bb'$.
		\end{enumerate}		
	\end{proof}
		
		\begin{figure}[ht]
			
			\begin{tikzpicture}
				\draw[thick] (0,1)--(0,-1.5);
				\draw[fill] (0,0) circle [radius=0.05];
				\node [right] at (0,0) {$\qq$};
				\draw[fill=blue] (1.2,0.7) circle [radius=0.05];
				\node [left] at (1.2,0.7) {$\bb$};
				
				\draw[fill=green] (2.5,-0.4) circle [radius=0.05];
				\node [left] at (2.5,-0.4) {$\bb$};
				
				\node [above left] at (0,0.8) {$\sL$};
				\node at (-1,0.7) {$\sK'$};
				\draw[pattern=north west lines, pattern color=red](0,0) -- (-0.6,0.6) --(-1.5,0.4)--(-2,0)--(-0.3,-.6)--cycle;
				\draw[dotted] (0,0) circle [radius=0.3];
				\draw[thick] (1,-1.5)--(-2,0);
				\draw[dashed] (-1.5,-1.5) -- (2,1);
				\draw[dashed] (-1.5,-2.05) -- (2,0.45);
				
				\draw[fill=red] (0,-1) circle [radius=0.1];
				\node [right] at (0,-1) {$\ba$};
				
			\end{tikzpicture}
			\caption{A reduction to the two dimensional case. The hyperplanes $\sH$ and $\sH_\ba$ are dotted in the figure. There are two points $\bb$ representing the two cases in the proof.}
			\label{fig:2d}
		\end{figure}
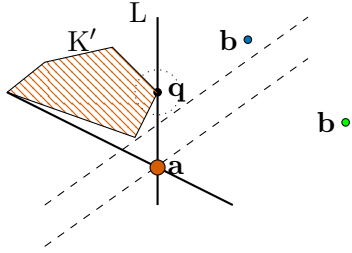

	\section{Proof of Theorem \ref{thm:main}}

	We first use polarity to transform the problem.
	
	\begin{definition}\label{def:coloring}
		Let $m>1$ be an integer.
		A convex body $\sQ$ has a \emph{proper $m-$coloring} if
		$\partial\sQ$ can be covered with sets $\cC_1,\dots,\cC_m$ such that
		\begin{enumerate}
			\item $\interior(\cC_i)\cap \interior(\cC_j)=\0$ for each $1\leq i<j\leq m$.
			\item For every subset $\0\subsetneq S\subsetneq[m]$ the set $\bigcup_{j\in S}\cC_j$ is either visible or covisible.
			If visible, then the set of witness points has nonempty interior. 
		\end{enumerate}
	\end{definition}
	\begin{remark}\label{rem:ball}
		Notice that in particular each set $\cC_i$ is visible/covisible.
		It follows that $\cC_i$ is homeomorphic to a $(d-1)$-dimensional ball.
	\end{remark}

	Given a properly colored convex body $\sQ$ we call the set $\textrm{Rainbow}(\sQ)=\bigcap_{[m]}\cC_i$ its \textbf{rainbow} set.
	Recall that $\mathcal{T}(\bK)$ is the set of all hyperplanes tangent to $\bK$ and with all bodies on the same side.
	
	\begin{proposition}\label{prop:polar}
		Let $\bK=\{\sK_1,\dots,\sK_m\}\subset \cK^d$ be a strictly separated family of convex bodies.
		There exists a $d$-dimensional convex body $\sQ$ with a proper $m-$coloring such that 
		$\mathcal{T}(\bK)$ is homeomorphic to $\textrm{Rainbow}(\sQ)$.
	\end{proposition}
	
	\begin{proof}
		Given a family $\bK=\{\sK_1,\dots,\sK_m\}\subset \cK^d$ of strongly separated convex bodies in $\Rr^d$, we consider its convex hull $\sK=\conv(\bK)$ which has nonempty interior.
		By translating if necessary we assume that $\mathbf{0}$ is in the interior of $\sK$.
		We consider the polar body $\sQ:=\sK^\circ$.
		
		Every point $\xx\in\partial\sK$ induces the supporting hyperplane $\{\yy\in\Rr^d~:~\langle\yy,\xx\rangle=1\}$ on $\sQ$.
		We define $\cC_i\subset\partial\sQ$ to be the union of the intersections of $\sQ$ with all supporting hyperplanes induced points in $\sK_i$.
		
		
		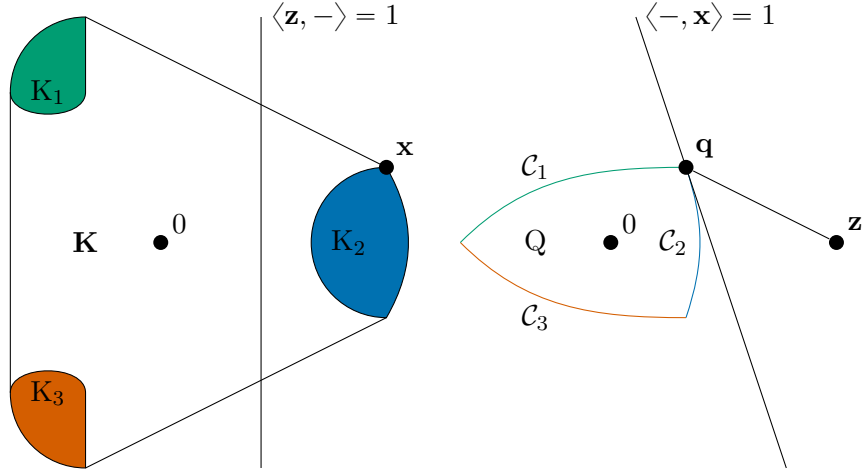
\begin{figure}
			\centering
			\begin{tikzpicture}
				\node[coordinate] at (2,0) (k21){};
				\node[coordinate] at (3,1) (k22){};
				\node[coordinate] at (3,-1) (k23){};
				\draw[fill=blue] (k21) to[out=90,in=180] (k22) to[out=300, in=60] (k23) to[out=180,in=270] (k21);
				\node at (2.5,0) {$\sK_2$};
				\node[circle,fill=black,inner sep=2pt] at (k22) {};
				\node[above right of = k22, node distance=10pt] {$\xx$};

				\node[coordinate] at (-1,2) (k11){};
				\node[coordinate] at (-2,2) (k12){};
				\node[coordinate] at (-1,3) (k13){};
				\draw[fill=bluishgreen] (k11) to[out=270, in=270] (k12) to[out=90, in=180] (k13) to (k11);
				\node at (-1.5,2) {$\sK_1$};
				
				\node[coordinate] at (-1,-2) (k31){};
				\node[coordinate] at (-2,-2) (k32){};
				\node[coordinate] at (-1,-3) (k33){};
				\draw[fill=red] (k31) to[out=90, in=90] (k32) to[out=270, in=180] (k33) to (k31);
				\node at (-1.5,-2) {$\sK_3$};
				
				\draw (4/3,3) -- (4/3,-3);
				\node at (2.33,3) {$\langle \zz, - \rangle=1$};
				
				\node[circle,fill=black,inner sep=2pt] at (0,0) (o){};
				\node[above right of = o, node distance=10pt] {$0$};

				\draw (k22) -- (k13);
				\draw (k23) -- (k33);
				\draw (k12) -- (k32);

				\node at (-1,0) {$\bK$};
			\end{tikzpicture}
			\quad
			\begin{tikzpicture}[scale=1]
				\node[circle,fill=black,inner sep=2pt] at (0,0) (o){};
				\node[above right of = o, node distance=10pt] {$0$};
				\node[coordinate] at (1,1) (c12){};
				\node[coordinate] at (1,-1) (c23){};
				\node[coordinate] at (-2,0) (c13){};
				\draw[blue] (c12) to[out=288.5, in=71.5] (c23);
				\draw[red] (c23) to[out=180,in=315] (c13);
				\draw[bluishgreen] (c13) to[out=45,in=180] (c12);
				
				\node[circle,fill=black,inner sep=2pt] at (3,0) (z){};
				\node[above right of = z, node distance=10pt] {$\zz$};
				
				\node at (-1,0) {$\sQ$};
				
				\node[circle,fill=black,inner sep=2pt] at (1,1) (q){};
				\node[above right of = q, node distance=10pt] {$\qq$};
				
				\draw (7/3,-3) -- (1/3,3);
				\node at (1.33,3) {$\langle -, \xx \rangle=1$};
				
				\draw (q) -- (z);

				\node at (-1,1) {$\cC_1$};
				\node at (5/6,0) {$\cC_2$};
				\node at (-1,-1) {$\cC_3$};
			\end{tikzpicture}
			\caption{Illustration of the polarity in Proposition \ref{prop:polar}.}
			\label{fig:my_label}
		\end{figure}
		
		We claim that the strong separation for $\bK$ implies that the union of $\cC_i$ for $i \in I$ is either visible or covisible:
		Let $\zz$ be a vector such that the hyperplane $\{\xx\in\Rr^d~:~\langle\zz,\xx\rangle=1\}$ strictly separates colors $I$ from the complement.
		We have that 
		\begin{equation}\label{eq:open}
			\langle\zz,\xx\rangle<1, \forall\xx\in\bigcup_{I^c}\sK_i,\quad \langle\zz,\xx\rangle>1, \forall\xx\in\bigcup_{I}\sK_i.
		\end{equation}
		\begin{enumerate}
			\item Let $\qq\in\cC_i\subset\sQ$ with $i\in I$.
			By definition there exists $\xx\in\sK_i$, such that $\langle \qq,\xx\rangle=1$.
			Since $\langle \pp,\xx\rangle\leq1$ for all $\pp\in\sQ$ but $\langle \zz,\xx\rangle>1$, the hyperplane $\{\yy\in\Rr^d~:~\langle\yy,\xx\rangle=1\}$ certifies that $\qq$ is visible from $\zz$ according to Lemma \ref{lem:visibility}.
			\item Let $\qq\notin\cC_i\subset\sQ$ for every $i\in I$.
			For any supporting hyperplane of $\sQ$ at $\qq$ we have an equation $\langle \qq,\xx\rangle=1$ with $\xx\notin \bigcup_I\cC_i$.
			For each of them we have $\langle \pp,\xx\rangle\leq1$ for all $\pp\in\sQ$ and $\langle \zz,\xx\rangle<1$, so by Lemma \ref{lem:visibility} then $\qq$ is not visible from $\zz$.
		\end{enumerate}
		In conclusion we have that $\vis_\zz(\sQ)=\bigcup_I\cC_i$.
		Equation \ref{eq:open} is an open condition of $\zz$, hence the set of witnesses has non empty interior. 
		Finally, since elements of $\mathcal{T}(\bK)$ are tangent to all $\sK_i$, the dual of these hyperplanes are points which lie in each $\cC_i$, the definition of $\textrm{Rainbow}(\sQ)$.
	\end{proof}
	
	\begin{remark}\label{rem:polytopes}
		The case of polytopes is simpler.
		Given a strongly separated family $\bP$ of $m$ full dimensional polytopes in $\Rr^d$, let $\sP=\conv\{\bP\}$ be their convex hull.
		We think of vertices of the polytope $\sP_j$ as being of color $j$, so that the vertices of $\sP$ are colored with the set $[m]$ and we are interested in the faces that contain a vertex of every color.
		Let $\sQ=\sP^\circ$ be the polar of $\sP$.
		We color each facet with the color of the corresponding vertex.
		In this case we always have that the witness points for a visibility region has nonempty interior.
	\end{remark}
	
	We have not defined a proper $1-$coloring, since the whole boundary of $\sQ$ is neither visible nor covisible.
	The case $m=1$ in Theorem \ref{thm:main} is trivial, as any $\uu\in\mathbb{S}^{d-1}$ defines a tangent hyperplane.
	When $m>1$ Proposition \ref{prop:polar} reduces Theorem \ref{thm:main} to the following statement.
	
	\begin{theorem}
	 \label{thm:dual}
		Let $\sQ\in\cK^d$ with a proper $m-$coloring.
		Then its rainbow set is homeomorphic to $\mathbb{S}^{d-m}$.
	\end{theorem}
	
	We will prove Theorem \ref{thm:dual} by induction on $m$.
	We use Proposition \ref{prop:projection} to lower the dimension, so we first verify that the conditions apply to our set up.
	To ease notation se write $\witness_\sQ(A)$ for $\witness_\sQ(\bigcup_{a\in A}\cC_a)$ whenever $A\subset[m]$.

	\begin{lemma}
		\label{lem:condition}
		Let $\sQ\in\cK^d$ be a convex body with a proper $m-$coloring.
		Let $\ba\witness_\sQ(\{m\})$, and $\sH$ a separating hyperplane.
		For any  nonempty $J\subsetneq[m-1]$, let $I=J\cup\{m\}$.
		The set $\witness_\sQ(I)$ contains an open ball $\sB$ such that $\projection_{\ba,\sH}(\sB)$ is disjoint from $\projection_{\ba,\sH}(\cC_m)$.
	\end{lemma}

	Before the proof we note that it is \emph{not} true that the whole $\projection_{\ba,\sH}(\witness_\sQ(I))$ is disjoint from $\projection_{\ba,\sH}(\cC_m)$.
	
	\begin{proof}
		Without loss of generality assume that $1\notin J$ and after a projective transformation we can assume that $\cC_1$ is covisible, so that its complement is visible, in particular we can assume that $\bigcup_I \cC_i$ is visible.
		
		We use the description of $\witness_{\sQ}(I)$ in Lemma \ref{lem:halfspaces}.
		There exists a hyperplane $\sH$ that is tangent to a $\qq\in\interior(\cC_m)$ and also to $\witness_{\sQ}(I)$ (otherwise $\witness_{\sQ}(I)=\witness_{\sQ}(J)$).
		Consider a sufficiently small open ball $\sB$ in $\witness_{\sQ}(I)$ such that $d(\sB,\sH)<d(\ba,\sH)$.
		We claim that the line $\sL$ spanned by $\ba$ and any $\bb\in\sB$ do not intersect $\sK$ which would conclude the proof.
		
		Let $\sH'$ be a hyperplane tangent to a point $\pp\in\interior(\bigcup_J \cC_j)$ that is also tangent to $\witness_{\sQ}(I)$.
		The segment $[\ba,\bb]$ intersects $\sH'$ in a point $\bc$.
		So we analyze the line $\sL$ in two sections:
		\begin{enumerate}
			\item In the ray from $\bc$ containing $\bb$ its does not intersect $\sK$ because it stays on the opposite side of $\sK$ with respect to $\sH$.
			\item In the ray from $\bc$ containing $\ba$ it does not intersect $\sK$ for because it stays on the opposite side of $\sK$ with respect to $\sH'$.
		\end{enumerate}
	\end{proof}
	
	The following Proposition is the inductive step we need.
	
	\begin{proposition}\label{lem:step}
		Let $\sQ\in\cK^d$ with a proper $m-$coloring.
		There exists a convex body $\sQ'\subset\Rr^{d-1}$ with a proper $(m-1)-$ coloring such that $\textrm{Rainbow}(\sQ)$ is homeomorphic $\textrm{Rainbow}(\sQ')$.
	\end{proposition}

	\begin{proof}
		
		The set $\witness_{\sQ}(\{m\})$ contains an open set, so we can pick a witness point $\ba$ and hyperplane $\sH$ separating $\ba$ from $\sQ$ that are sufficiently generic to apply Proposition \ref{prop:projection}.

		The projection $\projection_{\ba,\sH}$ maps $\cC_m$ into a convex body $\sQ'$. We claim that $\sQ'$ has a proper $(m-1)-$ coloring.
		For $i \in \{1,\dots,m-1\}$ let $\cC'_i$ be the image of $\cC_m\cap\cC_i \subset \partial\cC_m$ under the projection $\projection_{\ba,\sH}$.

		\begin{enumerate}
			\item The interiors of $\cC'_i$ and $\cC'_j$ are disjoint because the interiors of $\cC_m\cap\cC_i$ and $\cC_m\cap\cC_j$ are disjoint.
			These intersections are equal to $
			\partial\cC_m\cap\partial \cC_i$ and $\partial\cC_m\cap\partial\cC_j.$
			If $\pp$ is a point in the intersection, then every open ball $\sB$ centered in $\pp$ intersects the interiors of $\cC_m,\cC_j,$ and $\cC_i$.
			This implies that $\pp$ lies on the boundary of $\cC_m\cup\cC_j$ and thus on the boundary of $\cC_m\cap\cC_j$, and the same for $i$.
			%

			\item Let $J\subsetneq[m-1]$ and $I=J\cup\{m\}$.
			By the same arguments as in the proof of Lemma \ref{lem:condition} we can assume that both sets are visible.
			By Lemma \ref{lem:condition} there exists an open ball of witness points $\bb$ of the set $J$, satisfying the hypotheses of Proposition \ref{prop:projection}.
			We have
			\[	\bigcup_J(\cC_i\cap\cC_m)=\left(\bigcup_J\cC_i\right)
			\cap\cC_m=\partial\left(\bigcup_J\cC_i\right)
			\cap\partial\cC_m,
			\]
			since the interiors are disjoint.
			Also note that
			\[
			\bigcup_J\cC_i=\closure\left(\vis_\bb(\sQ)\backslash\vis_\ba(\sQ)\right),\quad \cC_m=\vis_\ba(\sQ)
			\]
			so by the second part of Proposition \ref{prop:projection}, we have
			\begin{equation*}
				\qq\in\bigcup_J(\cC_i\cap\cC_m) \Longrightarrow\qq'\in \vis_{\bb'}(\sM),
			\end{equation*}
			where $\qq'$ and $\bb'$ are the projections of $\qq$ and $\bb$ respectively under $\projection_{\ba,\sH}$.
			It follows that $\bb'$ is a witness for the visibility(or covisibility) of $\bigcup_J \cC'_j$.
			Furthermore, by assumption the set of such $\bb$ had a non empty interior, so the image also have a nonempty interior.
		\end{enumerate}
		
		Finally, by the first part of Proposition \ref{prop:projection} the map $\projection_{\ba,\sH}$ is an homeomorphism on the boundary.
		This homeomorphism restritcs to an homeomorphism between each set of colors and thus between the rainbow sets.
	\end{proof}
	The second condition in Definition \ref{def:coloring} is used to ensure that $\ba$ can be chosen to be generic which is an important hypothesis of the Theorem \ref{thm:key} that we use in Proposition \ref{prop:projection}.

	\begin{remark}
		In the case where each convex body is a polytope we can find a projection center $\ba$ as follows:
		Continuing with Remark \ref{rem:polytopes}, the polar $\sQ$ is a polytope.
		Each facet of $\sQ$ has an assigned color in $[m]$ and it is the intersection of an affine linear hyperplane $\{\xx\in\Rr^d~:~\langle \uu,\xx\rangle=b\}$ with $\sP$.
		We can assume that $\langle \uu,\xx\rangle\leq b$ for all points in $\sP$.
		A point $\ba$ outside of $\sQ$ that sees only the facets of color $m$ is characterized by the following finite linear strict inequalities:
		\begin{itemize}
			\item $\langle \uu,\ba\rangle< b$ if the corresponding facet is not of color $m$.
			\item $\langle \uu,\ba\rangle> b$ if the corresponding facet is of color $m$.
		\end{itemize}
		To find such $\ba$ we need to find a solution of a finite system of linear inequalities.
		This problem is equivalent to solving a linear program \cite[Theorem 10.4]{schrijver1998theory}, so it can be solved efficiently using the simplex method or any other linear programming algorithm.
	\end{remark}
	
	We are now in position to prove our dual statement.

	\begin{proof}[Proof of Theorem \ref{thm:dual}]
		By using Proposition \ref{prop:polar} and Lemma \ref{lem:step} we can reduce to the case where $\sQ$ is a convex body in $\Rr^{d-m+2}$ with a proper 2-coloring.
		In this case by the visibility of one of the colors we have a generic point $\ba$ such that the boundary of $\projection_{\ba,\sH}(\sQ)$ is homeormophic to its rainbow set.
		But $\projection_{\ba,\sH}(\sQ)$ is a convex body in $\Rr^{d-m+1}$ so its boundary is homeomorphic to $\mathbb{S}^{d-m}$.
	\end{proof}

	\begin{remark}
		The case $m=d$ of Theorem \ref{thm:dual} is similar to the Knaster--Kuratowski--Mazurkiewicz Lemma \cite{kkm}.
		For related results see an extension by Shapley \cite{shapley1973balanced} with an alternative proof by Komiya \cite{komiya1994simple} and a recent generalization by Frick and Zerbib \cite{frick_zerbib}.
	\end{remark}
	
	\begin{example}
		Strong separation is crucial to the statement of Theorem \ref{thm:main} since otherwise we can have an arbitrary number of tangents.
		
		Consider an $N$-agon $\sQ$ in $\Rr^2$ with vertices in the unit circle.
		Embed $\Rr^2$ in $\Rr^3$ by setting the last coordinate equal to zero.
		We define the following family of convex bodies.
		Let
		\begin{enumerate}
			\item $\sP$ be the pyramid over $\sQ$ with apex $(0,0,1)$.
			\item $\sB_1$ be the unit ball centered at $(0,0,10)$.
			\item $\sB_2$ be the unit ball centered at $(0,0,-10)$.
		\end{enumerate} 
		
		The family $\bK=\{\sP,\sB_1,\sB_2\}$ has three disjoint convex bodies (but it is not strongly separated) with $N$ common tangents.
		This highlights a difference with the algebraic approach in \cite{sottile2005real}, where the number of common tangents to $d$ quadrics in $\Rr^d$ is either at most $2^d$ or infinite.
	\end{example}

	We briefly consider the situation where each individual body can be separated from the rest. The following example demonstrates that the result fails to hold.
	
	\begin{example}
		Let $\bA=\{\pp_1,\pp_2,\pp_3,\pp_4\}\subseteq \Rr^2$ be a set the four vertices of a square oriented cyclically, so that the diagonals are $[\pp_1,\pp_3]$ and $[\pp_2,\pp_4]$.
		
		Consider the family of polytopes $\bP=\{\sP_1,\sP_2,\sP_3,\sP_4\}\subseteq\cK^4$ where
		$\sP_1=\pp_1\times[-1,1]^2, \sP_2=\pp_2\times[-2,2]^2, \sP_3=\pp_3\times[-1,1]^2,$ and $\sP_4=\pp_4\times[-2,2]^2$.
		Every color can be separated from the rest by a hyperplane, but not all subsets can be separated.
		No facet of $\conv\{\bP\}$ contains points from each polytope in the family.
	\end{example}

	\begin{corollary}\label{cor:tangent}
		Let $\bK=\{\sK_1,\dots,\sK_d\}\subset \cK^d$ be a family of strongly separated convex bodies in $\Rr^d$.
		There exists exactly two hyperplanes tangent to each convex body and with all the bodies in the same side.
	\end{corollary}

	\begin{figure}[ht]
		\centering
		\includegraphics[scale=0.25]{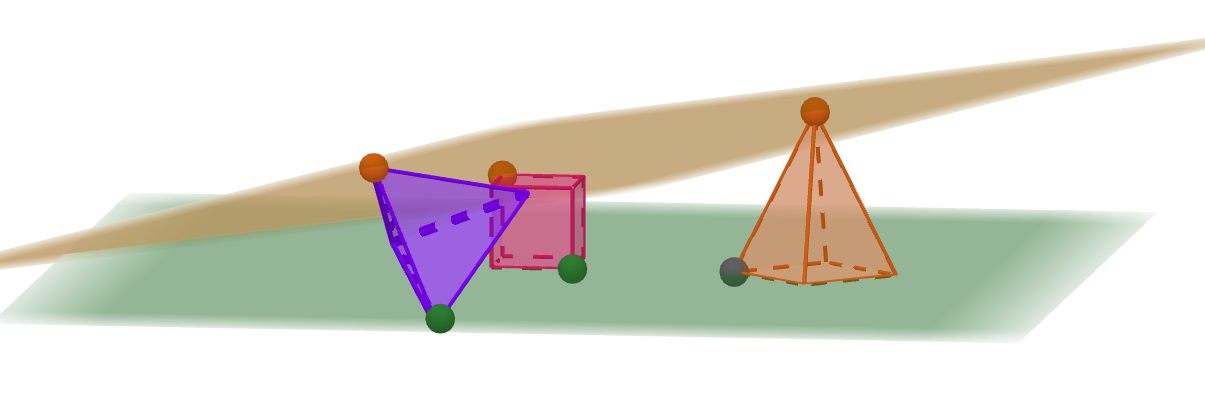}
		\caption{Example of Corollary \ref{cor:tangent} with three polytopes in $\Rr^3$}
		\label{fig:sandwich}
	\end{figure}
		
	\begin{theorem}\label{thm:polytopes}
	Let $\bP=\{\sP_1,\dots,\sP_m\}\subset \cK^d$ be a family of strongly separated full dimensional polytopes in $\Rr^d$ where $m\leq d$.
	The set $\mathcal{T}(\bP)$ is a polytopal complex combinatorially equivalent to the boundary of a $d-m+1$ dimensional polytope.
	\end{theorem}
	\begin{proof}
	As mentioned in Remark \ref{rem:polytopes}, the polar $\sQ$ is also polytopal and in the inductive step, Lemma \ref{lem:step}, we reduce the number of colors by a projection. Since projections of polytopes are again polytopes we obtain the desired conclusion.
	\end{proof}

	\section{Applications}
	
	\subsection{Common tangents to \texorpdfstring{$d$}{d} convex bodies in \texorpdfstring{$\Rr^d$}{Rd}}
	
	We now expand on the $m=d$ case of Theorem \ref{thm:main} with the goal of presenting an alternative Proof of Bisztriczky's Theorem.
	
	\begin{theorem}\label{thm:tangent}
		Let $\bK=\{\sK_1,\dots,\sK_{d}\}\in\cK^d$ be a family of strongly separated convex bodies in $\Rr^d$.
		For each unordered set partition $A\sqcup B=[d]$ there exists exactly two affine hyperplanes $\sH$ such that under an appropriate orientation of $\sH$:
		\begin{itemize}
			\item $\sH$ is tangent to each element of $\bK$.
			\item $\bigcup_{a\in A} \sK_a\subseteq \sH^\geq$.
			\item $\bigcup_{b\in B} \sK_b\subseteq \sH^\leq$.
		\end{itemize}
		Furthermore these affine hyperplanes are all different and thus there are $2^d$ tangent affine hyperplanes to the family $\bK$.
		
	\end{theorem}

	To prove the proposition, we first move into the linear setting, where the negation of a set changes its position with respect to some hyperplane.
	After we find a desired hyperplane in the linear setting, we return to the affine setting to finish the proof.

	\begin{proof}
		
		We linearize the $\sK_i$ to get a collection of cones $\sC_i$ which is linearly spanning and acyclic.
		Furthermore, the $\sC_i$ form a strongly separated family of cones\footnote{We call a family of cones strongly separated if their relative interiors are.}.
		
		Let $\overline{\bC}$ be the collection of cones which replaces $\sC_i$ with $-\sC_i$ for each $i \in B$.
		Since the $\sC_i$ are a strongly separated family, there is a hyperplane $\sH_\ba$ which separates $A$ and $B$.
		The linearization of $\sH_\ba$ proves the acyclicity of $\overline{\bC}$. Every element of $B$ has been negated, so every cone lies on the same side of the linearized hyperplane.
		For some generic partition $D \sqcup E = [d]$, by the strong separation of the $\sK_i$, there is a hyperplane $\sH_{D}$ separating $D \Delta B$ from $E \Delta B$, where $\Delta$ represents the symmetric difference.
		The linearization of $\sH_{D}$ separates $D$ and $E$ in $\overline{\bC}$, since each element of $B$ swapped parts within the partition.
		This shows that $\overline{\bC}$ is strongly separated as well.
		
		We apply the Corollary \ref{cor:tangent} to the affinization of $\overline{\bC}$ to get two affine hyperplanes that are tangent to each  color with all points of this affinization on one side of the hyperplanes.
		By linearization, we obtain linear hyperplanes which are again tangent to each colored cone, and all cones $\{\sC_{i}~:~i\in A\}\cup\{-\sC_{i}~:~i\in B\}$ are on the positive side.
		Finally, undoing the negation of the cones in $B$, and returning to the original affine setting, we have obtained two affine hyperplanes tangent to every $\sK_i$ and such that it separates sets $A$ and $B$.
		
		From the $2^{d-1}$ partitions of $[d]$ we obtain $2^d$ common tangent hyperplanes. 
		We simply need to conclude that all these hyperplanes are unique.
		Given two hyperplanes obtained from different partitions, there is a pair of bodies whose interiors are on a common side of one hyperplane, but separated by the other hyperplane, so no hyperplanes from different partitions can be the same.
		Since we already proved there are two distinct hyperplanes for each partition, there can be no repeated hyperplanes among the $2^d$ of them.
	\end{proof}
	\begin{example}
		In Figure \ref{fig:3d} we illustrate an example of Theorem \ref{thm:tangent}.
		\begin{figure}[ht]
			\centering
			\includegraphics[scale=0.25]{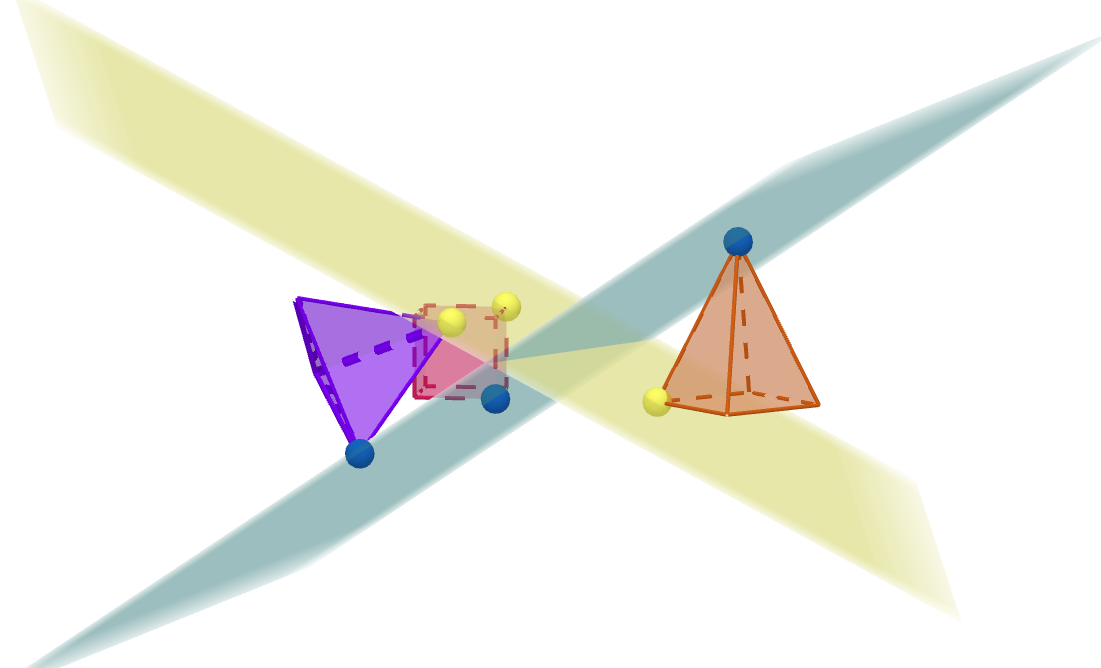}
			\caption{Two tangent hyperplanes separating the tetrahedron and the cube from the square pyramid.}
			\label{fig:3d}
		\end{figure}
	\end{example}

	We finish this section with an open question.
	
	\begin{question}
		Given two disjoint collections of partitions of $[n]$, when is there a family of $n$ convex bodies in $\Rr^n$, so that the first collection of partitions are all separated and none of the second collection of partitions are?
	\end{question}

	\subsection{Collection of \texorpdfstring{$d+1$}{d+1} convex bodies}
	
	We first relax the conditions of Theorem \ref{thm:tangent} a little bit.
	We say a family $\sS=\{\sS_1,\cdots,\sS_d\}$ of compact convex sets in $\Rr^d$ is affinely spanning if $\conv(\sS)$ is full dimensional.
	
	\begin{proposition}
		\label{prop:span}
		Let $\bS=\{\sS_1,\cdots,\sS_d\}$ be a strongly separated and affinely spanning family of compact convex sets in $\Rr^d$.
		For each set partition $A\sqcup B=[d]$ there exists exactly two affine hyperplanes $\sH$ such that:
		\begin{itemize}
			\item $\sH$ is tangent to each element of $\bK$.
			\item $\bigcup_{a\in A} \sK_a\subseteq \sH^\geq$.
			\item $\bigcup_{b\in B} \sK_b\subseteq \sH^\leq$.
		\end{itemize}
	\end{proposition}
	\begin{proof}
		At the start of the proof of Proposition \ref{prop:polar} we used the fact that the convex hull of a family of convex bodies is a convex body, that is it has a nonempty interior.
		This was needed to assume that $\mathbf{0}$ is in the interior which allowed us to take the polar body.
		With the extra assumption that $\conv(\bS)$ is full dimensional Theorem \ref{thm:main}  still holds: by using Proposition \ref{prop:polar} we can reduce it to Theorem \ref{thm:dual}.
		Then the proof of Theorem \ref{thm:tangent} applies in the present context without change, except that we cannot guarantee the hyperplanes are all different.
	\end{proof}
	\begin{example}
	As an example of the difference between Proposition \ref{prop:span} and Theorem \ref{thm:tangent} consider a triangle $\Delta\in\Rr^2$ and a disjoint point $\pp$.
	There exists two (as opposed to four) lines tangent to both of them
	\end{example}

	We cannot have a strongly separated family with $d+2$ or more convex sets in $\Rr^d$ since strong separation implies that their centroids are affinely independent.
	So we focus on the case with $d+1$ elements.

	\begin{proposition}\label{prop:d+1}
		Let $\bS=\{\sS_1,\dots,\sS_{d+1}\}$ be a family of strongly separated convex bodies of $\Rr^d$ and let $A\sqcup B=[d+1]$ a set partition, together with a special element $a\in A$.
		There exists a \emph{unique} hyperplane $\sH$ such that
		\begin{enumerate}
			\item $\sH$ is tangent to $\bS\setminus \sS_a$.
			\item $\bigcup_{i\in A} \sS_i\subseteq \sH^\geq$.
			\item $\bigcup_{i\in B} \sS_{i}\subseteq \sH^\leq$.
			\item $\sS_a\subseteq H^+$.
		\end{enumerate}
	\end{proposition}

	\begin{proof}
		Without loss of generality we assume $a=d+1$.
		Embed the family $\bS=\{\sS_1,\dots,\sS_{d+1}\}$ in $\Rr^{d+1}$ by using zero in the last coordinate, and additionally make a thickening of $\sS_{d+1}$: replace it by its Minkowski sum with the ball $\sB(\epsilon)$. If $\epsilon$ is small enough the strong separation still holds.
		The thickening ensures that the resulting family is affinely spanning in $\Rr^{d+1}$.
		
		Now we apply Proposition \ref{prop:span} with the sets $A,B$.
		We obtain two distinct hyperplanes $\sL_1$ and $\sL_2$ that are tangent to every set, including $\sS_{d+1}+\sB(\epsilon)$.
		We go down to $\Rr^d$ by intersecting with the hyperplane $\sL=\{ x\in\Rr^{d+1}~:~x_{d+1}=0\}$, to obtain two hyperplanes $\sH_1=\sL\cap\sL_1,\sH_2=\sL\cap\sL_2$ in $\Rr^d$ satisfying the conditions (1)--(4).
		
		To conclude the proof we must prove that actually $\sH_1=\sH_2$.
		We argue by contradiction and assume they are different.
		Running the same argument with $\sS_{d+1}$ on the $B$-side strictly we get at least one hyperplane $\sH_3$ in $\Rr^{d}$ satisfying (1)--(3) and (4) reversed.
		This hyperplane $\sH_3$ is necessarily different from $\sH_1$ and $\sH_2$ since the set $\sS_{d+1}$ lies on different sides with respect to the sets in $B$.
		But then the three hyperplanes $\sH_1,\sH_2,$ and $\sH_3$ are different and satisfy (2)--(3) with respect to $\{\sS_1,\dots,\sS_{d}\}$ contradicting Proposition \ref{prop:span}.
		This shows that $\sH_1=\sH_2$ concluding the uniqueness of $\sH$.
		
	\end{proof}

	\subsection{A different separation condition}
	
	The motivation for this paper was certain conditions that arose in \cite{cubes}, which used a different definition of separation.
	We say a family $\bK=\{\sK_1,\dots,\sK_{d+1}\}$ of convex sets in $\Rr^d$ is simplicially separated if
	\begin{itemize}
		\item[$(\star)$] The intersection of all simplices having a vertex on each set of the family is full dimensional.
	\end{itemize}
	Theorem \cite[Theorem 5.9]{cubes} states that if a family satisfy the condition, then intersection of all rainbow simplices is itself a simplex.
	The proof uses a version of Proposition \ref{prop:d+1} when $A$ is a singleton, but in that case the existence of the hyperplane is almost given by assumption and one need to check only uniqueness.
	\begin{proposition}
		Simplicial separation implies strong separation but the reverse is not true.
	\end{proposition}
	\begin{figure}[t]
		\centering
		\includegraphics[scale=0.6]{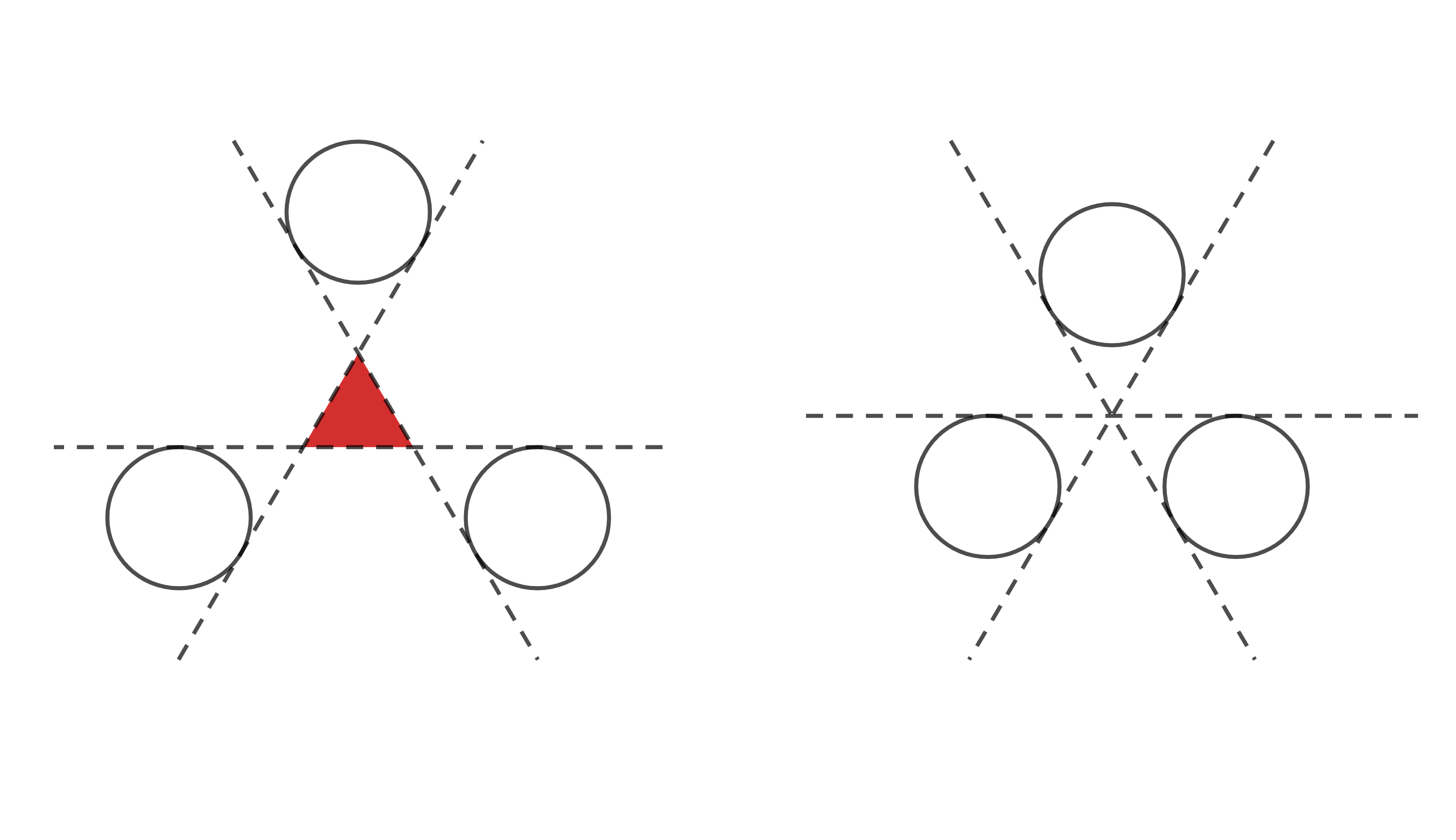}
		\caption{In the left figure the intersection of all rainbow triangles is highlighted. In the right figure the intersections of all rainbow triangles is a single point.}
		\label{fig:rainbowfail}
	\end{figure}
	\begin{proof}
		
		Simplicial separation implies that the intersection of all rainbow simplices is a simplex $\sS$.
		The simplex $\sS$ is described in \cite{cubes} as follows.
		For each color $i$ there exists an affine hyperplane $\sH_i$ such that $\sH_i$ is tangent to $\sK_j$ for $j\neq i$, $\bigcup_{j\neq i}\sK_j\subseteq \sH^\leq$ whereas $\sK_i\subseteq \sH^+$.
		The $d+1$ hyperplanes $\sH_i$ define the facets of $\sS$.
		We denote $\vv_i$ the vertex of $\sS$ not contained in $\sH_i$.
		
		Consider the affine hyperplane arrangement $\cH=\{\sH_1,\dots,\sH_{d+1}\}$ in $\Rr^d$.
		There is a unique bounded region, the simplex $\sS$, and $d+1$ pointed cones, one opposite to each vertex of $\sS$.	
		Each convex body $\sK_i$ is contained in the pointed cone opposite (with respect to $\sS$) to $\vv_i$.
		
		Now consider any partition $A\sqcup B=[d+1]$.
		We have $\dim \text{aff. span}\{\vv_a:a\in A\}+\dim \text{aff. span}\{\vv_b:b\in B\}=d-1$, so their sum is an affine hyperplane $\sH$.
		If we translate $\sH$ so that it contains the barycenter $\bb$ of $\sS$, then we obtain an affine hyperplane that does not instersect any of the pointy regions.
		This hyperplane is a strict separator for $A,B$.
		
		On the other hand strong separation does not imply Property $(\star)$, see for example Figure \ref{fig:rainbowfail}
		
	\end{proof}

	\bibliography{biblio}
	\bibliographystyle{plain}

\end{document}